\documentclass[12pt]{amsart}
\usepackage{amssymb,amsmath}
\usepackage{geometry}    
\usepackage{mathrsfs}

\usepackage{graphicx}

\usepackage{vmargin}

\usepackage{tikz}  

\setmargrb{1in}{1in}{1in}{1in}

\newtheorem{theorem}{Theorem}[section]
\newtheorem{lemma}[theorem]{Lemma}
\newtheorem{proposition}[theorem]{Proposition}

\theoremstyle{definition}
\newtheorem{definition}{Definition}
\newtheorem{remark}{Remark}

\newtheorem{conjecture}{Conjecture}

\usepackage{xcolor}

\begin{document}
	
	\title[$\Phi$ badly approximable matrices have full dimension]{The set of $\Phi$ badly approximable matrices has full Hausdorff dimension}
	
	\author{Johannes Schleischitz}

	\thanks{Middle East Technical University, Northern Cyprus Campus, Kalkanli, G\"uzelyurt \\
		johannes@metu.edu.tr ; jschleischitz@outlook.com}

\begin{abstract}
    Recently Koivusalo, Levesley, Ward and Zhang introduced the set of simultaneously $\Phi$-badly approximable real vectors of $\mathbb{R}^m$ with respect to an approximation function $\Phi$, and
    determined its Hausdorff dimension
    for the special class of power functions $\Phi(t)=t^{-\tau}$. We refine this by naturally extending the formula to arbitrary decreasing functions, in terms of the lower order of $1/\Phi$ at infinity. 
    We also provide an alternative, rather mild condition on $\Phi$ for this conclusion. Moreover, our results apply in the general matrix setting, and we establish an according formula for packing dimension as well. 
    Thereby we also complement a recent refinement by Bandi and de Saxc\'e
    on the smaller set of exact approximation with respect to $\Phi$. Our basic tool 
is (a uniform variant of) the variational principle by Das, Fishman, Simmons, Urba\'nski. We also prove some new lower estimates regarding the set of exact approximation order in the matrix setting, which are sharp in special instances. For this we combine 
the result by Bandi and de Saxc\'e with a method developed by Moshchevitin.
\end{abstract}

\maketitle

{\footnotesize{

		Math Subject Classification 2020: 11H06, 11J13}}

\section{Introduction: Sets of exact and bad approximation}  \label{intro}

Let $m,n$ be positive integers and $\Phi: \mathbb{N}\to (0,\infty)$ be a function
with the property that 
\begin{equation}  \label{eq:2}
    \Phi(t)\le t^{-n/m}, \qquad t\ge 1.
\end{equation}
Note that we do not assume monotonicity.
Define the set of $\Phi$-approximable matrices $W(\Phi)$ to be the set of real $m\times n$ matrices $A$ such that
\[
\Vert A \textbf{q}-\textbf{p}\Vert \leq \Phi(\Vert \textbf{q}\Vert)
\]
for infinitely many integer vectors $\textbf{p}\in \mathbb{Z}^{m}, \textbf{q}\in\mathbb{Z}^{n}$, where for ease of notation
$\Vert.\Vert$ denotes 
maximum norm on both $\mathbb{R}^m$ and $\mathbb{R}^n$.
Let $\tau=\tau(\Phi)$ be the lower order of $1/\Phi$ at infinity
\[
\tau= \liminf_{t\to\infty} -\frac{\log \Phi(t) }{\log t} \in [\frac{n}{m},\infty].
\]
The most natural instances are power functions 
$\Phi(t)=t^{-\tau}$ with $\tau\ge n/m$, where
the right hand side is constant.
If $\tau$ is finite, in this case for simplicity
let $W_{\tau}=W(t\to t^{-\tau})$, 
which is essentially the set of matrices approximable
to order $\tau$. For $\tau=\infty$,
accordingly let $W_{\infty}= \cap_{\tau\in \mathbb{R} } W_{\tau}$ be the set of $m\times n$ Liouville matrices. Notice that if $\tau=n/m$, the set $W_{\tau}$
becomes the entire space of real $m\times n$ matrices
by a well-known variant of Dirichlet's Theorem. This also explains why our restriction \eqref{eq:2} is natural, although strictly speaking
the set $W(\Phi)$ does not necessarily 
consist of all real $m\times n$ matrices if $\tau<n/m$
for a general $\Phi$ (there are counterexamples if the {\em upper} order of $1/\Phi$ at infinity
exceeds $n/m$ and $\Phi$ fluctuates fast; it is however true for any monotonic $\Phi$).
Define the set
\[
Exact(\Phi)= W(\Phi)\setminus \bigcup_{c>1} W(\Phi/c) 
\]
of ``exactly'' $\Phi$ approximable $m\times n$ matrices.
Indeed, it consists of the matrices that are $\Phi$-approximable, but not $\Phi/c$-approximable
no matter how close $c>1$ is to $1$. It is suggestive 
to believe in the following

\begin{conjecture}  \label{konsch}
	The set $Exact(\Phi)$ has the same Hausdorff dimension
	as $W_{\tau}$ as soon as 
	\begin{equation} \label{eq:200}
	\Phi(t)=o(t^{-n/m}), \qquad t\to\infty.
	\end{equation}
\end{conjecture}

We will occasionally for brevity refer to the dimension
as ``full'' in this case.
For $m=n=1$ the conjecture is essentially the content of~\cite[Problems~1 \& 2]{bugmo}. 
Possibly \eqref{eq:200} can be relaxed in some ways, however \eqref{eq:2}
is insufficient at least for $m=n=1$ due to well-known facts about the Lagrange spectrum, as already pointed out in~\cite{bugmo}.
For $m=n=1$ and decreasing $\Phi$, 
in fact under a certain weaker assumption $(\ast)$ on $\Phi$,
the conjecture was verified by 
Bugeaud and Moreira~\cite[Theorem~3]{bugmo}, see 
also the preceding work by
Bugeaud~\cite{bugeaud, bugeaud2}, both dimensions being $2/(\tau+1)$. 
Only very recently, during the process of finishing the present paper, this was
extended to the case of simultaneous approximation to $m$ real numbers,
i.e. $n=1$, by Bandi and 
de Saxc\'e~\cite{bandi}, the full dimension being $(m+1)/(\tau+1)$.
The latter paper~\cite{bandi} again imposed $\Phi$ to be decreasing and satisfying \eqref{eq:200}.
While it is likely
that the method in~\cite{bandi} can be extended to the
general matrix setting by similar arguments,
in all other
cases of $m,n$ the conjecture remains yet unknown even assuming
monotonic $\Phi$. 
The case of non-monotonic $\Phi$ is not 
completely understood even for $m=n=1$.

Prior to~\cite{bandi},
in preprint~\cite{ward} from autumn 2023 as well, it is shown that again for $n=1$ and the special case $\Phi(t)=t^{-\tau}$, $\tau>1/m$,
the larger set of $\Phi$-badly approximable vectors defined by
\[
Bad(\Phi)= \bigcup_{c>1} W(\Phi)\setminus W(\Phi/c)=
W(\Phi)\setminus \bigcap_{c>1} W(\Phi/c)
\]
has full Hausdorff dimension $(m+1)/(\tau+1)$, a weaker result. The latter paper~\cite{ward} in turn improved on results from~\cite{vela}.
For any $\Phi$ we have the obvious inclusions
\begin{equation} \label{eq:RF}
Exact(\Phi)\subseteq Bad(\Phi)\subseteq 
\bigcap_{\epsilon>0} W_{\tau-\epsilon}.
\end{equation}
The methods
from~\cite{bandi, ward} are rather different.
We refer to~\cite{bandi, bugmo, ward} for a more 
comprehensive history of exact approximation and
more relevant literature. We want to add
that in~\cite[Corollary~7]{resm}
it is shown that again for $n=1$, for arbitrary $\Phi$ with $\tau>\tau_m$ large enough, the {\em packing} dimension of $Exact(\Phi)$ equals $m$, hence is full literally. Hereby  suitable $\tau_m$ can be explicitly determined, decreasing to the limit $(3+\sqrt{5})/2=2.6180\ldots$ with initial term 
$\tau_1=(5+\sqrt{17})/2=4.5615\ldots$. 
It is natural to believe in an analogue of Conjecture~\ref{konsch} for packing dimension. The method from~\cite{resm} is very different from any other previous work on the topic, including the present paper.

\noindent\textbf{Acknowledgements:} The paper
was started during a visit for an International Visitor Program at the University of Sydney. The author thanks the University of Sydney for the generous support. 

\section{New results}

% For the purpose of this paper, motivated by the sets
% $Bad(\Phi)$ from~\cite{ward} introduced in~\S~\ref{intro},
% we define the twisted set of ``symmetrically $\Phi$-badly approximable'' matrices by
% \[
% SBad(\Phi)= \bigcup_{c>1} W(c\Phi)\setminus W(\Phi/c).
% \]
% This notion is new. This set is slightly larger than the previously studied ones discussed in~\S~\ref{intro},
% indeed for any $\Phi$ we have the obvious inclusions
% \begin{equation} \label{eq:RF}
% Exact(\Phi)\subseteq Bad(\Phi)\subseteq SBad(\Phi)\subseteq 
% \bigcap_{\epsilon>0} W_{\tau-\epsilon}.
% \end{equation}
% Note however that if $\Phi(t)=t^{-n/m}$ then $Bad(\Phi)=SBad(\Phi)$ are both just the usual set of badly approximable $m\times n$ matrices, whereas
% $Exact(\Phi)$ is supposedly empty (this is certainly the case for $m=n=1$ by Hurwitz Theorem).

In this paper, we show that under some rather mild additional regularity assumption on $\Phi$, the Hausdorff dimension of $Bad(\Phi)$ is full, i.e. the same as for $W_{\tau}$. Thereby we
properly extend the result in~\cite{ward} and complement~\cite{bandi}. 
% in two aspects: for more general functions $\Phi$, and our result is applicable to the general matrix setting.
Moreover, we obtain
an analogous result for packing dimension.
As an alternative to assuming $\Phi$ monotonic, a 
widely used simplifying condition used
in~\cite{bandi, ward}, 
the following rather mild condition (C1) suffices: For any $c>1$, there exists $d=d(c)>1$ such that
\begin{equation} \label{eq:1}  \tag{C1}
c^{-1}t\le \tilde{t}\le ct \quad \Longrightarrow \quad
d^{-1}\Phi(t)\le \Phi(\tilde{t})\le d\Phi(t),
\end{equation}
for any positive integers $t, \tilde{t}$. We now turn
towards our main result, a discussion on condition \eqref{eq:1} is given below.
Denote by $\dim_H$
and $\dim_P$ the Hausdorff and packing dimensions of a set respectively. 
We show

\begin{theorem}  \label{thm1}
    Let $\Phi$  satisfy \eqref{eq:2} and
    additionally assume that either condition \eqref{eq:1} holds
    or $\Phi$ is decreasing. Then 
    \begin{equation}  \label{eq:beggin}
    \dim_H(Bad(\Phi)) = \dim_H(W_{\tau})= (n-1)m+\frac{m+n}{1+\tau},
    \end{equation}
    and
     \begin{equation}  \label{eq:leggin}
    \dim_P(Bad(\Phi)) = \dim_P(W_{\tau})= mn.
    \end{equation}
\end{theorem}

%\begin{remark}
%For monotonic $\Phi$ and $n=1$, the %claim is implied by the stronger %results from~\cite{bandi}. Our proof in %this case
%will be independent.
%\end{remark}

An alternative sufficient condition for \eqref{eq:leggin}, applicable for the smaller sets $Exact(\Phi)$, will be given in Theorem~\ref{mo} below.
The dimensions of $W_{\tau}$ in the theorem are well-known, see~\cite{ward}
for a reference. For the more interesting identities involving $Bad(\Phi)$,
the upper estimates are rather obvious, in view of the right inclusion of \eqref{eq:RF} it suffices to take the limit of the (continuous) dimension formulas
for the sets $W_{\tau-\epsilon}$ as $\epsilon\to 0$.
As in~\cite{bandi, ward}, the reverse lower bounds are the 
substantial result. 

We discuss \eqref{eq:1} in more detail now.
It can be readily checked that we only need to assume the implication holds for some $c>1$. 
Moreover, if \eqref{eq:1} holds for some pair 
$c_0, d_0=c_0^{\rho_0}$,
via partitioning $(1,\infty)$ into intervals $(1,c_0], [c_0,c_0^2), [c_0^2,c_0^4),[c_0^4,c_0^8),\ldots$, 
it can be seen that for any $c>1$ it holds for $d=\max\{ d_0, c^{2\rho_0} \}$. In particular if \eqref{eq:1} holds then we can assume $\rho(c)= \log d/\log c\ll_{\Phi,\epsilon} 1$ bounded on any interval $c\in [1+\epsilon,\infty)$.
Essentially, condition \eqref{eq:1} says that $\Phi$ does not
deviate too fast on short intervals.
It is in particular true for all power functions $t\to t^{-\tau}$, as
treated in~\cite{ward}.
It is easily seen that the assumptions $\Phi$ being monotonic and \eqref{eq:1} are independent. Moreover \eqref{eq:1} it is independent from 
the condition $(\ast)$ from~\cite[\S~2]{bugmo} for $m=n=1$ already mentioned in~\S~\ref{intro},
which can be stated as $\Phi(ct)\le 4 c^{\rho}\Phi(t)$ for any 
$c>1$ and some absolute $\rho\in\mathbb{R}$, uniformly 
in $t\ge t_0$. 
Hence $\Phi$ may decay arbitrarily fast, so for example $\Phi(t)=e^{-t}$ satisfies $(\ast)$
but not \eqref{eq:1}. On the other hand, for example the piecewise defined function $\Phi(t)=2^N t^{-3}$ for $t\in[2^N,2^{N+1})$, $N\in \mathbb{N}$, satisfies \eqref{eq:1} but not $(\ast)$. Note that $\rho(c)\to \infty$ as $c\to 1^{+}$ for the latter.
%Note however that as in~\cite{bandi}, paper~\cite{bugmo} shows %``full'' dimension for the smaller sets $Exact(\Phi)$, a %stronger implication.

We compare Theorem~\ref{thm1} with the claims from~\cite{bandi, bugmo, ward, resm} discussed in~\S~\ref{intro}.
Claim~\eqref{eq:beggin} refines results 
from~\cite{bandi, ward} in two aspects: Firstly, we can deal with the general matrix setup instead of just simultaneous approximation $n=1$.
Secondly, our result applies to a wider class of approximation functions that are
possibly not decreasing
(especially~\cite{ward} restricted to $t\to t^{-\tau}$, $\tau>1/m$). Our above comparison of \eqref{eq:1} versus condition $(\ast)$ from~\cite{bugmo} shows that even for $m=n=1$ there
is some new information derived from \eqref{eq:beggin}.
% A disadvantage is that we have to consider the larger set $SBad(\Phi)$ for technical reasons. 
In particular
for $n=1$ and decreasing $\Phi$ formula \eqref{eq:beggin} is implied by the
stronger claims from~\cite{bandi} (however we give an independent proof).
It may however be viewed as a disadvantage
that our proof relies on the heavy machinery in~\cite{dfsu}. 
Also, as an artifact of these results
in~\cite{dfsu}, we cannot provide 
a fixed
constant $c>1$ such that the full dimension claim holds for the set
$W(\Phi)\setminus W(\Phi/c)\subseteq Bad(\Phi)$, which
is possible in~\cite{ward}.
The packing dimension
result \eqref{eq:leggin} can be regarded a variant of~\cite[Corollary~7]{resm}.
Advantages are that it applies for matrix setting 
and only the weaker natural bound $\tau\ge n/m$ on $\tau(\Phi)$ is needed. However, we must assume \eqref{eq:1} or monotonic $\Phi$ instead
and cannot provide any information on
the smaller sets $Exact(\Phi)$ treated in~\cite{resm}.
In terms of proofs, our strategy below
will use similar arguments as in~\cite{bandi}, also based on parametric geometry of numbers (template theory), however with some twists.
Most notably, we directly apply the 
variational principle from~\cite{dfsu},
in place of the Cantor set construction
in~\cite{bandi} that allows for treating
the smaller set $Exact(\Phi)$.

In fact our proof of Theorem~\ref{thm1} below will show the following refined claim. As usual $a\asymp_{.} b$ means
$a\ll_{.} b\ll_{.} a$, where $a\ll_{.} b$ in turn means $a\le cb$
for some constant $c>0$ depending on the subscript variables only, while $a,b$ may depend on other parameters as well.
If there are no subscript variables the constants are absolute.

\begin{theorem}  \label{Thm2}
    Let $\Phi$ be as in Theorem~\ref{thm1}. Let $\chi: \mathbb{N}\to \mathbb{R}$ by any function/sequence (increasing arbitrarily fast). Then the metrical conclusions of Theorem~\ref{thm1} hold for the set of real $m\times n$ matrices $A$ for which the following two properties hold:
    \begin{itemize}
        \item There exists a sequence $(t_k^{\prime})_{k\ge 1}$ of integers, depending on $A$, such that
        \[
        t_k^{\prime} > \chi(t_k^{\prime}),
        \]
        and
        \[
        \Phi(t_k^{\prime})= t_k^{\prime -\tau+o(1)}, \qquad k\to\infty,
        \]
        and
        \[
        \Vert \textbf{q}_k\Vert = t_k^{\prime}, \qquad
        \Phi(t_k^{\prime})\ll_{A} \Vert A \textbf{q}_k-\textbf{p}_k\Vert < \Phi(t_k^{\prime}), \qquad k\ge 1,
        \]
        holds for some integer vector $(\textbf{p}_k, \textbf{q}_k)\in \mathbb{Z}^{m}\times (\mathbb{Z}^{n}\setminus \{ \boldsymbol{0}\})$.
        \item For any $(\textbf{p}, \textbf{q})\in \mathbb{Z}^{m}\times (\mathbb{Z}^{n}\setminus \{ \boldsymbol{0}\})$ not among the 
        $(\textbf{p}_k, \textbf{q}_k)$ we have
        \[
        \Vert A \textbf{q}-\textbf{p}\Vert \gg_A \Vert \textbf{q}\Vert^{-n/m}, \qquad k\ge 1.
        \]
    \end{itemize}
\end{theorem}

In view of \eqref{eq:2} this clearly implies Theorem~\ref{thm1}. In order to apply~\cite{dfsu},
the implied constants in the theorem will tend to infinity to derive the full dimension results. It may be true that 
analogous extensions of the results stated in~\cite{bandi}
can be derived from their method as well, as the proofs show
substantial similarities.

Finally, we want to state some new claims for the sets $Exact(\Phi)$.
From~\cite{bandi, resm}
when combined with a method originating in~\cite{ngm}, under certain conditions
on $\Phi$ we derive a non-trivial lower bound for the Hausdorff dimension and full packing dimension in the matrix setting.
Let 
\[
\theta=\theta_{m,n}:=\frac{\max\{ n+1,m+2\}}{m} > 1.
\]
Then our result reads as follows.

\begin{theorem}  \label{mo}
	The following claims hold:
	\begin{itemize}
	\item[(a)] Let $\Phi$ induce $\tau>\theta$ and be decreasing. Then
	  \begin{equation} \label{eq:1111}
	\dim_H(Exact(\Phi))\ge m(n-1)+\frac{m+1}{1+\tau}.
	\end{equation}
	\item[(b)] Let $\Phi$ induce $\tau>\max\{ \theta, \frac{5+\sqrt{17}}{2} \}$ (not necessarily decreasing). Then
	\begin{equation} \label{eq:222}
	\dim_H(Exact(\Phi))\ge m(n-1)
	\end{equation}
	and
		  \begin{equation} \label{eq:333}
		\dim_P(Exact(\Phi))= mn.
		\end{equation}
	\end{itemize}
\end{theorem} 

\begin{remark}  \label{rema}
	The weaker property $(\ast)$
	from~\cite{bugmo} discussed in~\S~\ref{intro} instead of $\Phi$ decreasing, together with $\tau>\theta$, suffice to settle the 
	intermediate bound
	\[
	\dim_H(Exact(\Phi))\ge m(n-1)+\frac{2}{\tau+1},
	\]
	in between \eqref{eq:1111}, \eqref{eq:222}.
	In particular for $\tau=\infty$ we get full
	dimension, refining \eqref{eq:beggin} from Theorem~\ref{thm1} in this case. See \S~\ref{RR} below
	for a proof.
\end{remark}

For $\tau=\infty$, the Hausdorff dimension in \eqref{eq:1111} and \eqref{eq:222} is full, a refinement of Theorem~\ref{thm1}.
The condition $\tau>\theta>n/m$ is more restrictive than \eqref{eq:2} and \eqref{eq:200}.
It can be slightly relaxed as the proof of Lemma~\ref{lehre} below shows, however we cannot decrease the bound for $\tau$ with our method,
in particular \eqref{eq:200} is insufficient in both claims.
For $n=1$ estimate \eqref{eq:1111}
is implied by~\cite{bandi}. The value $(5+\sqrt{17})/2$ in the condition for \eqref{eq:222},	\eqref{eq:333} derived from~\cite{resm}
can be relaxed if $m>1$.

\section{Combined graphs, templates and the variational principle}

For $1\le j\le m+n$ and $A\in \mathbb{R}^{m\times n}$,
and a parameter $Q\ge 1$,
let $\lambda_j(Q)$ be the $j$-th successive minimum
with respect to the parametric convex body
\[
K(Q)= [-Q,Q]^n \times [-Q^{-n/m},Q^{-n/m}]^m \subseteq \mathbb{R}^{m+n}
\]
and the lattice 
\[
\{  (\textbf{q},A\textbf{q}-\textbf{p}):\; \textbf{q}=(q_1,\ldots,q_n)\in \mathbb{Z}^n,\; \textbf{p}=(p_1,\ldots,p_m)\in \mathbb{Z}^m  \}\subseteq \mathbb{R}^{m+n}.
\]
Further derive continuous piecewise linear functions
\[
h_j(q)= \log ( \lambda_j(e^q)), \qquad 1\le j\le m+n,\; q>0. 
\]
Thereby every real $m\times n$ matrix $A$ gives rise to successive minima functions $\textbf{h}=(h_1,\ldots,h_{m+n})$
on $(0,\infty)$.
We call the union of the graphs $(q,h_j(q))$, $q>0$ 
over $1\le j\le m+n$ in $\mathbb{R}^2$ the combined graph associated to $A$. The above definition is formally slightly different than the formulation using homogeneous dynamics from~\cite{dfsu}. We rather follow
the setup of Schmidt and Summerer~\cite{ss}
but for general $m,n$, however the two formalisms are easily verified to be equivalent.
Define the trajectory associated to $(\textbf{p},\textbf{q})\in \mathbb{Z}^{m}\times(\mathbb{Z}^{n}\setminus\{\textbf{0}\})$
as 
\begin{equation} \label{eq:traje}
T_{\textbf{p}, \textbf{q}}(q)= \max \left\{  \log \Vert \textbf{q}\Vert - \frac{1}{n}q \;,\; 
\log \Vert A\textbf{q}-\textbf{p}\Vert + \frac{1}{m}q \right\}.
\end{equation}
Then $h_1(q)$ is the minimum of all trajectories
\[
h_1(q)=\min_{\textbf{p}, \textbf{q}} T_{\textbf{p}, \textbf{q}}(q)= \min_{\textbf{p}, \textbf{q}} \max \left\{  \log \Vert \textbf{q}\Vert - \frac{1}{n}q \;,\; 
\log \Vert A\textbf{q}-\textbf{p}\Vert + \frac{1}{m}q \right\}
\]
with minimum taken over all integer vectors $(\textbf{p},\textbf{q})\in \mathbb{Z}^{m}\times(\mathbb{Z}^{n}\setminus\{\textbf{0}\})$. More generally,
for $1\le j\le m+n$, we have
\[
h_j(q)= \min \max_{1\le i\le j} T_{\textbf{p}_i, \textbf{q}_i}(q)
\]
with the minimum taken over all tuples of $j$ linearly independent integer vectors $(\textbf{p}_i, \textbf{q}_i)\in \mathbb{Z}^{m}\times (\mathbb{Z}^n\setminus\{ \textbf{0}\})$.
Since any $h_j$ locally coincides with some trajectory,
it is clear that it is indeed piecewise linear with slopes  among $\{-1/n,1/m\}$.

We next introduce templates.
Following~\cite[Definition~4.1]{dfsu}, an $m\times n$-template consists of piecewise linear functions 
$\textbf{f}(q)=(f_1(q), \ldots, f_{m+n}(q))$ with a certain finite
set of slopes, including $\{ -1/n,1/m\}$ as smallest and largest elements, and several other restrictions, 
we omit stating the details here for brevity.
We highlight that the sum of the $f_j(q)$, thus also of the derivatives where defined, vanishes on $q\in(0,\infty)$.

Each template gives rise to a local contraction rate $\delta(q)\in [0,mn]$ for each $q>0$ where $\textbf{f}$ is differentiable, depending on the slope
vector $\textbf{f}^{\prime}(q)$ at $q$, again we omit details
here to be found in~\cite{dfsu}. The lower and upper limits 
\[
\underline{\delta}(\textbf{f})=\liminf_{Q\to\infty} \frac{ \int_{0}^{Q} \delta(q) \; dq }{Q},\qquad \overline{\delta}(\textbf{f})=\limsup_{Q\to\infty} \frac{ \int_{0}^{Q} \delta(q)\; dq }{Q}
\]
are called lower and upper contraction rate of the template.
Denote $\mathcal{T}_{m,n}$ the set of $m\times n$ templates. 

\begin{definition}
We say a combined graph $\textbf{h}$ is
$T>0$ close to a template $\textbf{f}\in\mathcal{T}_{m,n}$ if 
\[
\sup_{q>0} \max_{1\le j\le m+n} \vert f_j(q)-h_j(q)\vert\le T.
\]
% We say $\textbf{h}$ is $O(1)$ close to $\textbf{f}\in \mathcal{T}_{m,n}$ if 
% this is true for some $T$, that is
% \[
% \sup_{q>0} \max_{1\le j\le m+n} \vert f_j(q)-h_j(q)\vert<\infty.
% \]
\end{definition}
We identify $A\in\mathbb{R}^{m\times n}\cong \mathbb{R}^{mn}$
giving rise to the natural $mn$-dimensional Lebesgue measure
on the sets of $m\times n$ matrices.
A partial claim of the landmark result~\cite[Theorem~4.7]{dfsu} can be formulated as follows:

% \begin{theorem}[Das, Fishman, Simmons, Urba\'nski]  \label{dfsu}
%     If $\textbf{f}\in \mathcal{T}_{m,n}$ is a template with lower and upper contraction rates $\underline{\delta}(\textbf{f})$ and $\overline{\delta}(\textbf{f})$, then the set of real matrices $A\in\mathbb{R}^{m\times n}$ inducing 
%     a combined graph $\textbf{h}(A)$ that is $O(1)$ close to $\textbf{f}$
%     has Hausdorff dimension at least 
%     $\underline{\delta}(\textbf{f})$ and 
%     packing dimension at least
%     $\overline{\delta}(\textbf{f})$.
% \end{theorem}

% It will be convenient to also apply the following uniform variant, a partial claim of
%~\cite[Theorem~3.15]{dfsu}.

\begin{theorem}[Das, Fishman, Simmons, Urba\'nski]  \label{dfsu02}

Let
$\textbf{f}\in \mathcal{T}_{m,n}$ is a template with lower and upper contraction rates $\underline{\delta}(\textbf{f})$ and $\overline{\delta}(\textbf{f})$.
    Given $\varepsilon>0$, there is some absolute $T=T(\varepsilon)>0$ independent of $\textbf{f}$ so that
    the set of 
    real matrices $A\in\mathbb{R}^{m\times n}$ inducing 
    a combined graph $\textbf{h}(A)$ that is $T$ close to $\textbf{f}$
    has Hausdorff dimension at least 
    $\underline{\delta}(\textbf{f})-\varepsilon$ and 
    packing dimension at least
    $\overline{\delta}(\textbf{f})-\varepsilon$.
\end{theorem}

\section{Proof of Theorem~\ref{thm1} upon assumption \eqref{eq:1} }

%We first note that we can restrict to the case
%\begin{equation}  \label{eq:toe}
%%\lim_{t\to\infty} t^{n/m} \Phi(t) = 0.
%\end{equation}
%Otherwise $SBad(\Phi)$ contains all badly approximable
%matrices in the classical sense, i.e. it contains %$Bad_{m,n}=SBad(t\to t^{-n/m})$, which is known to have full %dimension. Indeed, 
%choosing an increasing subsequence 
%of positive integers $t_k$ where $\Phi(t_k)\gg t_k^{-n/m}$
%and Dirichlet's Theorem show that any matrix lies in some %$W(c\Phi)$ for large enough $c$.
%
%$SBad(\Phi)$ contains $W(c\cdot (t\to t^{-n/m}))$ for any %large $c$, on the
%other hand by $\Phi(t)< t^{-n/m}$ it is clear that no element %of it belongs to $W(c^{-1}\cdot (t\to t^{-n/m}))$ for large %enough $c$.BAUSTELLE!!!!!!!!!!

%We will assume $\tau<\infty$ for simplicity,
%otherwise just small adjustments of our method are required.

\subsection{Construction of a template}  \label{3.1}

Let $\Phi$ satisfy \eqref{eq:2}, \eqref{eq:1}.
We first construct a template $\textbf{f}=\textbf{f}(\Phi)$, in a 
similar way to the proof of~\cite[Theorem~3.17]{dfsu}.
Take a lacunary sequence of integers $t_k>1$ such that 
\begin{equation}  \label{eq:aut}
    \lim_{k\to\infty} -\frac{\log \Phi(t_k)}{ \log t_k}= \tau, 
\end{equation}
where by lacunary we mean
\begin{equation}  \label{eq:lacu}
\lim_{k\to\infty} \frac{\log t_{k+1} }{ \log t_k} = \infty.
\end{equation}
Note $\tau=\infty$ is possible.
%(Some stronger increase property depending on $\Phi$ is needed %if $\tau=\infty$.)
We restrict to the class $\mathcal{C}\subseteq \mathcal{F}_{m,n}$ of templates $\textbf{f}=(f_1,\ldots,f_{m+n})$ with the property
\[
f_2(q)=f_3(q)=\cdots=f_{m+n}(q)=-\frac{f_1(q)}{m+n-1}, \qquad q>0.
\]
Moreover we exclude the template with all component functions constant $0$ for all $q\ge q_0$. Then for each such template
$\textbf{f}\in \mathcal{C}$, the first component
$f_1$ has slopes among $\{-1/n,0,1/m\}$, and we have 
\[
f_1(q)=f_2(q)=\cdots=f_{m+n}(q)=0
\] 
for $q$ on a sequence of pairwise disjoint closed intervals 
$I_1=[a_1,b_1],I_2=[a_2,b_2],\ldots$ on the first axis,
where $a_1<b_1<a_2<b_2<\ldots$.

For simplicity let $c_{k}=a_{k+1}$ so that $J_k=(b_k,c_{k})$ denotes the open interval between $I_k$ and $I_{k+1}$. Then in each $J_k$, the function $f_1$ has a unique local minimum, say at $q_k\in J_k$, and $f_1$ decreases with slope $-1/n$ on $(b_k,q_k)$ and increases with slope $1/m$ on the intervals $(q_k,c_{k})$. The remaining $f_j$, $2\le j\le m+n$, have accordingly slopes
$n^{-1}/(m+n-1)$ and $-m^{-1}/(m+n-1)$ on the respective intervals.
We choose our concrete template $\textbf{f}\in \mathcal{C}$ so that these local minima of $f_1$ are at positions $(q_k,f_1(q_k))$ with 
\begin{equation} \label{eq:iden}
    q_k=(\log t_k - \log \Phi(t_k)) \frac{1}{ m^{-1}+n^{-1} }, 
\qquad f_1(q_k)= \log t_k - \frac{1}{n} q_k-\log \Delta,
\end{equation}
for $\Delta\ge 1$ to be chosen later.
This resembles the minimum of a fictive trajectory $T_{\textbf{p},\textbf{q}}$ 
with $\Vert \textbf{q}\Vert=t_k$ and $\Vert A\textbf{q}-\textbf{p}\Vert= \Phi(\Vert \textbf{q}\Vert)/\Delta$. 
It is easily verified that
the $q_k$, and thus the $t_k$, uniquely determine the template and
\begin{equation}  \label{eq:preis}
b_k=q_k+nf_1(q_k),\qquad c_{k}=q_k-mf_1(q_k)
\end{equation}
hold for $k\ge 1$.
Recall that $f_1(q_k)<0$, so indeed $b_k<c_k$.\\

\begin{figure}
\centering
\begin{tikzpicture}
\draw[black, ->] (0,0) -- (10.5, 0) node[anchor=west]{$q$};
\draw[black, ->] (0, 0) -- (0, 2.5) node[anchor=south]{$\textbf{f}(q)$};
\filldraw[black] (1.5,0) circle (1pt) node[anchor=north]{$b_k$};
\draw[black] (1.5, 0) -- (3.3, -1.2) node[pos=0.4 , below]{$-\frac{1}{n}$};
\draw[black] (4.5 , 0) -- (3.3, -1.2) node[pos=0.4 , below]{$+\frac{1}{m}$};
\draw[black, thick] (1.5, 0) -- (3.3, 0.75);
%node[pos=0.3 , above]{$+\frac{1}{n(m+n-1)}$};
\draw[black, thick] (4.5, 0) -- (3.3, 0.75);
\draw[black, ultra thick] (4.5, 0) -- (7, 0) node[pos=0.5 , below]{$0$};
\filldraw[black] (4.5,0) circle (1pt) node[anchor=north]{$c_k$};
\filldraw[black] (3.3,0) circle (1pt) node[anchor=north]{$q_k$};
\filldraw[black] (7,0) circle (1pt) node[anchor=north]{$b_{k+1}$};
\draw[black] (7, 0) -- (9.5, -1.6) node[pos=0.4 , below]{$-\frac{1}{n}$};
\draw[black, thick] (7, 0) -- (9.5, 1.2) node[anchor=south]{$f_{2}=\cdots=f_{m+n}$};
\draw[black, thick] (10, 0.8) -- (9.5, 1.2);
\filldraw[black] (9.5,0) circle (1pt) node[anchor=north]{$q_{k+1}$};
\draw[black] (10, -1.1) -- (9.5, -1.6) node[anchor=north]{$f_{1}$} ;
%\caption{M1} \label{fig:M1}

\coordinate [label=left:$+\frac{1}{n(m+n-1)}$] (A) at (3,0.8);
\coordinate [label=left:$+\frac{1}{n(m+n-1)}$] (B) at (7.96,0.8);
\coordinate [label=left:$-\frac{1}{m(m+n-1)}$] (C) at (5.6,0.7);
\coordinate [label=left:$f_1$] (D) at (3.45,-1.55);
%\coordinate [label=left:$f_{m+n}$] (E) at (8.8,2);

\end{tikzpicture}

Figure 1: Sketch of template with slopes

\end{figure}

We claim two statements that when combined almost directly yield the theorem in the version where we assume \eqref{eq:1}. Recall the contraction rates of a template from the last section. 

\begin{lemma}  \label{l2}
Let $\Phi$ satisfy \eqref{eq:2}, \eqref{eq:1} and $T>0$.
    Any matrix $A$ whose induced combined graph 
    $\textbf{h}=\textbf{h}_A$ is $T$
    close to $\textbf{f}$ constructed above belongs to $Bad(\Phi)$. 
\end{lemma}

\begin{lemma}  \label{l1}
    The template $\textbf{f}$ above has lower and upper average 
    contraction rate
    \[
    \underline{\delta}(\textbf{f})=(n-1)m+\frac{m+n}{1+\tau},\qquad \underline{\delta}(\textbf{f})=mn,
    \]
    thus via Theorem~\ref{dfsu02} induces
    a set of matrices with combined graph $O(1)$ 
    close to it
    with the desired Hausdorff and packing dimensions.
\end{lemma}

The latter lemma is again proved very similarly to~\cite[Theorem~3.17]{dfsu}
so there is not much originality to it, the proof of the former lemma is however more involved.
We prove the lemmas below.

\subsection{Proof of Lemma~\ref{l2} }

The claim follows from the next geometrically very intuitive observation,
whose exact proof however turns out rather lengthy.

\begin{lemma} \label{dl}
Let $\textbf{f}\in \mathcal{T}_{m,n}$ be the template constructed in~\S\ref{3.1}.
Let $T>0$ be a parameter and $\textbf{h}$ be the combined graph associated to some $A\in \mathbb{R}^{m\times n}$ that is
$T$ close to $\textbf{f}$. Then there exists an absolute constant $C=C(T)>0$, 
independent of $\Delta$ from \eqref{eq:iden}, so that 
\begin{itemize}
    \item for each integer $k\ge k_0$, there is a unique local minimum $(r_k,h_1(r_k))$ 
of $h_1$ that is $C$ close to the local minimum $(q_k,f_1(q_k))$ of $f_1$, i.e.
$r_k\in [q_k-C, q_k+C]$ and 
$h_1(r_k)\in [f_1(q_k)-C, f_1(q_k)+C]$.
\item any other local minimum $(r,h_1(r))$ of $h_1$ satisfies
$h_1(r)\ge -C$.
\end{itemize}

\end{lemma}

%We should notice 
%that the points $(r_k,h_1(r_k))$ in the lemma depend on the %choice of $A$.

\begin{proof}
Recall $f_1$ decays with slope $-1/n$ on intervals $(b_k, q_k)$, increases with slope $1/m$ on $(q_k,c_{k})$ and is $0$ outside these intervals. Since $\sup_{q>0} \vert f_1(q)-h_1(q)\vert \le T$
we obviously have that $h_1(q)<-T$ implies $b_k<q<c_{k}$.
Moreover, since $f_2(q)\ge 0$ everywhere and $\vert f_2(q)-h_2(q)\vert \le T$
on $q\in (0,\infty)$, the graph of $h_2$ has no points
with second coordinate less than $-T$. Since any local 
maximum of $h_1$ is a local minimum of $h_2$, it follows
that there are no local maxima of $h_1$ with function
value smaller than $-T$. 
Hence we have the following: On any interval $[b,c]$ where 
on the interior $(b,c)$ we have $h_1(q)< -T$ 
and at the boundary $h_1(b)=h_1(c)=-T$,
the function $h_1(q)$ must 
be equal to some fixed trajectory $T_{\textbf{p},\textbf{q}}$. This means $h_1(q)$ first
decreases starting from $b$ with slope $-1/n$ up to some local minimum attained at some $r>b$, and then starts
increasing with slope $1/m$ until it reaches $c>r>b$ with $h_1(c)=-T$ again. 
We clearly have $[b,c]\subseteq [b_k,c_{k}]$ for some $k$ as outside $f_1(q)=0$ and thus
$h_1(q)\ge -T$, contradiction to $h_1(q)<-T$ in a right resp. left neighborhood of $b$ resp. $c$. 
More precisely, both $b, c$ must lie in one of the two subintervals (for now this may or may not be the same for $b$ and $c$)
\begin{equation} \label{eq:5}
[b_k, b_k+2nT] , \qquad [c_{k}- 2mT , c_{k}]
\end{equation}
since in the remaining interval $(b_k+2nT,c_{k}-2mT)$ we have
$f_1(q)< -2T$ and thus $h_1(q)<-T$, but $h_1(b)=h_1(c)=-T$.

We may assume the two subintervals in \eqref{eq:5} are disjoint for all large $k$. Indeed, this is in particular true if
\begin{equation}  \label{eq:4}
    \lim_{k\to\infty} f_1(q_k)= -\infty,
\end{equation}
as then by \eqref{eq:preis} we have $c_{k}-b_k=-(m+n)f_1(q_k)\to\infty$.
However \eqref{eq:4}
is in turn always true if $t_k^{n/m}\Phi(t_k)\to 0$
as $k\to\infty$, as can be checked via \eqref{eq:iden}. Now if otherwise $\Phi(t_k)\gg t_k^{-n/m}$ on a subsequence of $k$, then
by \eqref{eq:2} it is easily seen that $Bad(\Phi)\supseteq Bad(t\to t^{-n/m})=Bad$ contains the usual set of badly approximable real $m\times n$ matrices (in fact our assumptions
on $\Phi$ imply identity). Then
the claim of Theorem \ref{thm1}
follows from the well-known result by Schmidt~\cite{schmidt} (or alternatively derived from the variational principle) that 
the set of badly approximable matrices has full Hausdorff
dimension $mn$.

From the disjointness, it is obvious from the above description of $h_1$ on $[b, c]$ that for each $k$ there can be at most one pair $b,c$ as above with the property
\begin{equation}  \label{eq:111}
b_k \le b \le b_k+2nT , \qquad c_{k}- 2mT \le  c \le c_{k}.
\end{equation}
We consider $k$ fixed now and 
we write $r_k$ for the local minimum 
position $r$, which indeed will be the value stated
in Lemma~\ref{dl}.
In particular
$\vert b-b_k\vert \ll_{T,m,n} 1$ and $\vert c-c_{k}\vert\ll_{T,m,n} 1$ uniformly in $k$.

For all other pairs where either both $b,c$ are in 
$[b_k, b_k+2nT]$ or both are in $[c_{k}- 2mT , c_{k}]$,
since the slope of $h_1$ is bounded between $-1/n$ and $1/m$ and $h_1(b)=h_1(c)=-T$,
it is clear that $h_1$ is uniformly bounded from below 
on $[b,c]$ in terms of $m,n,T$ only. In fact 
from intersecting the line containing $(b,-T)$ with slope
$-1/n$ with the line containing $(c,-T)$ with slope $1/m$, recalling at $r=r_k$ is the local minimum of $h_1$ within $[b,c]$, we get the bound
\[
\min_{q\in [b,c]} h_1(q) = h_1(r) = -T-\frac{c-b}{m+n}\ge -T- \frac{2T\max\{m,n\}}{m+n}=:-C_1.
\]
Now we consider the potential pair as in \eqref{eq:111}.
In fact this pair must exist for any large $k$ 
by \eqref{eq:4}. Indeed, for large $k$ we have 
$f_1(q_k)< \min\{ -2T, -C_1-T\}$ hence
$h_1(q_k)\le f_1(q_k)+T< \min\{ -T, -C_1\}$ 
and we observed that only for pairs in \eqref{eq:111} this may happen, for large $k$.
Now for this pair in \eqref{eq:111}
we have
\[
\min_{q\in [b,c]} h_1(q) = h_1(r_k) = -T-\frac{c-b}{m+n}= -T- \frac{c_{k}-b_k}{m+n} + \frac{E}{m+n}
\]
with $0\le E \le 2(m+n)T$,
so by \eqref{eq:preis} we conclude
\[
\min_{q\in [b,c]} h_1(q) = h_1(r_k) \in [- \frac{c_{k}-b_k}{m+n} -  T, - \frac{c_{k}-b_k}{m+n} +  T]= [f_1(q_k) - T, f_1(q_k)+T].
\]
In other words 
\begin{equation}   \label{eq:rach}
    | h_1(r_k) - f_1(q_k)| \le T.
\end{equation}
Moreover, $r_k$ can be determined via solution
of $-T-(r_k-b)/n=-T-(c-r_k)/m$ as 
\[
r_k= \frac{ bm+cn }{m+n}.
\]
Similarly 
\[
q_k= \frac{ b_km+c_{k}n }{m+n}.
\]
Hence by \eqref{eq:111} we get
\[
\vert r_k-q_k\vert \le \frac{4mnT}{m+n}=: C_2 \ll_{m,n} T.
\]
So together with \eqref{eq:rach}, indeed the points $(q_k, f_1(q_k))$ and
$(r_k,h_1(r_k))$ are at most $C_3:=\max\{ C_2, T\} \ll_{m,n} T$ apart with respect to maximum norm on $\mathbb{R}^2$, proving the first
claim of the lemma. The second is obvious
as from the arguments above any other local minimum $\tilde{q}\notin \{ r_1,r_2,\ldots\}$ of $h_1(q)$ satisfies $h_1(\tilde{q})\ge \min\{ -T, -C_1\}=-C_1$, so we can put
$C=\max\{ C_1, C_3\}$ to satisfy both claims. Note that $C$ is independent of $\Delta$ as requested.
    \end{proof}

Given $\varepsilon>0$, start with $T=T_0(\varepsilon)$ induced by Theorem~\ref{dfsu02}.
With $C=C(T)$ from Lemma~\ref{dl},
take the square $[q_k-C, q_k+C]\times [f_1(q_k)-C, f_1(q_k)+C]$ and let $X_k\subseteq \mathbb{R}$ be the interval
obtained from its projection to the second axis
along the line with slope $-1/n$ and $Y_k\subseteq \mathbb{R}$
the interval induced by its projection to the second axis along the line with slope $1/m$. This means
we identify $X_k, Y_k$ with real intervals by dropping
the first coordinate $0$.
Then by the formula of trajectories \eqref{eq:traje}
and as at the minimum the two component functions
within the maximum coincide,
the integer vectors $(\textbf{p}_k, \textbf{q}_k)\in \mathbb{Z}^{m}\times (\mathbb{Z}^{n}\setminus \{\textbf{0}\})$ inducing $(r_k,h_1(r_k))$, depending on $\textbf{h}$ and thus the choice of $A$,
satisfy
\[
\log \Vert \textbf{q}_k\Vert \in X_k, \qquad 
\log \Vert A\textbf{q}_k-\textbf{p}_k\Vert \in Y_k. 
\]
Since it is easy to see that $X_k\subseteq [f_1(q_k)+q_k/n-D, f_1(q_k)+q_k/n+D]$
and 
$Y_k\subseteq [f_1(q_k)-q_k/m-D, f_1(q_k)-q_k/m+D]$ for some absolute $D=D(C)>0$ independent of $k$ and $\Delta$, taking exponential map we see that uniformly in $k$
\begin{equation} \label{eq:NEU}
\Vert \textbf{q}_k\Vert \asymp_{D} e^{f_1(q_k)+q_k/n}= t_k, \qquad 
\Vert A\textbf{q}_k-\textbf{p}_k\Vert \asymp_{D} e^{f_1(q_k)-q_k/m} =\frac{\Phi(t_k)}{\Delta}.
\end{equation}
For the identities we have used \eqref{eq:iden}, and the
latter requires a short calculation.
Recall $T$ is fixed
and so are $C,D$, independent of $\Delta$. Hence, by property 
\eqref{eq:1}, choosing $\Delta$ large enough we can guarantee that $\Phi(t_k)/\Delta<K^{-1}\Phi(\Vert \textbf{q}_k\Vert)$ for all $k\ge 1$, 
% for some absolute $d=d(c)$ and all $k\ge 1$, 
where $K=K(D)=K(T)$
is the implied upper constant in the right formula of \eqref{eq:NEU}. 
%In other words
%\[
%\Phi(t_k) \asymp_D  \Phi(\Vert\textbf{q}%\Vert),
%\]
Thus
\begin{equation} \label{eq:NN}
\Phi(\Vert \textbf{q}_k\Vert)\ll_{D}  \Vert A\textbf{q}_k-\textbf{p}_k\Vert < 
K \frac{\Phi(t_k)}{\Delta} <
K\cdot (K^{-1}\Phi(\Vert \textbf{q}_k\Vert))= \Phi(\Vert \textbf{q}_k\Vert). 
\end{equation}
Again since $C$ and thus $D$ is fixed once $T$ is chosen, we may deduce that 
\[
\Phi(\Vert \textbf{q}_k\Vert)\ll_{T}  \Vert A\textbf{q}_k-\textbf{p}_k\Vert < \Phi(\Vert \textbf{q}_k\Vert). 
\]
The right estimate shows that any $A$ whose combined graph is $T$ close to
the template $\textbf{f}$ lies in $W(\Phi)$. 
On the other hand, the left estimate shows that the $(\textbf{p}_k, \textbf{q}_k)\in \mathbb{Z}^{m}\times \mathbb{Z}^{n}$ inducing $(r_k,h_1(r_k))$ do not induce considerably better approximations, 
so for $A$ to lie in all $W(\Phi/c)$, $c>1$ we may 
restrict to considering other integer vectors $(\textbf{p}, \textbf{q})\in \mathbb{Z}^{m}\times \mathbb{Z}^{n}$.
However, by Lemma~\ref{dl} any $(\textbf{p}, \textbf{q})\in \mathbb{Z}^{m}\times (\mathbb{Z}^{n}\setminus \{\textbf{0}\})$
not of the form $(\textbf{p}_k, \textbf{q}_k)$ induces a local minimum point $(r,h_1(r))$ with $h_1(r)\ge -C$, which by a short calculation
and \eqref{eq:2} means
\[
\Vert A\textbf{q}-\textbf{p}\Vert \gg_{C} \Vert \textbf{q}\Vert^{-n/m} \ge \Phi(\Vert \textbf{q}\Vert),
\]
and again since $C$ is fixed once $T$ is chosen also
\[
\Vert A\textbf{q}-\textbf{p}\Vert \gg_{T}  \Phi(\Vert \textbf{q}\Vert).
\]
This argument shows that any $A$ with combined graph $T$ close 
to $\textbf{f}$ cannot lie in all $W(\Phi/c)$, $c>1$.
Combining our observations, indeed any such $A$ belongs to $Bad(\Phi)$.

\subsection{Proof of Lemma~\ref{l1} }

As announced, our proof is very similar
to~\cite[Theorem~3.17]{dfsu}.
We use the notation 
in~\cite[\S~4.1]{dfsu}.
Following~\cite[Definition 4.5]{dfsu},
the local contraction rate of a template $\textbf{f}$ 
at $q>0$ equals the cardinality of the set
\begin{equation}  \label{eq:ii}
    \{ (i_{+},i_{-}): i_{+}\in S_{+}, i_{-}\in S_{-},\;\; i_{+}<i_{-}\}, 
\end{equation}
where $S_{-}, S_{+}$ depend on $q$ and
form a partition of $\{1,2,\ldots,m+n\}$, we omit stating the full intricate precise definition from~\cite[Definition 4.5]{dfsu}.
In the next proposition we determine the contraction rates on $(0,\infty)$. Denote $Y_{\mathbb{Z}}= Y\cap \mathbb{Z}$
for a set $Y\subseteq \mathbb{R}$.

\begin{proposition}  \label{preprep}
   The local contraction rate of the template constructed in~\S~\ref{3.1}
   is given as
$$
\delta(q)= \begin{cases}
mn, \qquad\qquad \text{if } \; q\notin \cup (b_k, q_k)  \\
mn-m, \qquad \text{if } \;  q\in \cup (b_k, q_k).
\end{cases}
$$
\end{proposition}

\begin{proof}
In intervals of the form
$(a_k,b_k)=(c_{k-1}, b_k)$, i.e. the interior of $I_k$, where all $f_j$ vanish,
the only interval of equivalence as 
defined in~\cite[Definition 4.5]{dfsu} 
is $(0,mn]_{\mathbb{Z}}=\{ 1,2,\ldots , m+n\}$ and
$S_{-}=\{m+1,m+2,\ldots,m+n\}$ and $S_{+}=\{1,2,\ldots,m\}$, so the cardinality of
\eqref{eq:ii} is 
$mn$, as large as possible.
On intervals $(b_k,q_{k})$,
the intervals of equivalence are $(0,1]_{\mathbb{Z}}=\{1\}$ 
and $(1,mn]_{\mathbb{Z}}=\{ 2,3,\ldots,m+n\}$ and 
$S_{-}=\{1,m+2,m+3,\ldots,m+n\}$ 
and $S_{+}=\{2,3,\ldots,m+1\}$, so the set 
\eqref{eq:ii}
has cardinality $mn-m$.
Finally, on intervals $(q_k,c_k)$, the intervals
of equivalence are 
$(0,m+n-1]_{\mathbb{Z}}=\{ 1,2,\ldots,m+n-1\}$ and the singleton $(m+n-1,m+n]_{\mathbb{Z}}=\{ m+n\}$,
but then again
$S_{-}=\{m+1,m+2,\ldots,m+n\}$ and 
$S_{+}=\{1,2,\ldots,m\}$. So the cardinality of
\eqref{eq:ii} is again $mn$.
\end{proof}

By the proposition, it is clear that the local minima of the average contraction rate are taken at the values $q_k$, and the local maxima at values $b_k$, so that
\begin{equation}  \label{eq:expr}
\underline{\delta}(\textbf{f})= \liminf_{k\to\infty} \frac{\int_{0}^{q_k} \delta(q)\; dq }{q_k},  \qquad
\overline{\delta}(\textbf{f})= \limsup_{k\to\infty} \
\frac{\int_{0}^{b_k} \delta(q) \; dq }{b_k}.
\end{equation}
From \eqref{eq:aut} and \eqref{eq:iden} if $\tau<\infty$ we see
\begin{equation}  \label{eq:frant}
    q_k = \frac{\tau+1+o(1)}{ m^{-1}+n^{-1} }\cdot \log t_k, \qquad k\to\infty,
\end{equation}
and similarly we can just assume the first factor grows if $\tau=\infty$. In either case \eqref{eq:lacu} implies
\begin{equation} \label{eq:franzl}
    \lim_{k\to\infty} \frac{q_{k+1}}{q_k}=  \infty.
\end{equation}
On the other hand, since by \eqref{eq:traje} and Dirichlet's Theorem obviously $-1/n\le f_1(q_k)/q_k\le 0$, relations 
\eqref{eq:iden}, \eqref{eq:preis} imply $b_k\asymp c_k\asymp q_k$, with implied constants depending on $\Delta$. Combining with \eqref{eq:franzl} yields $b_{k+1}/q_k\to \infty$ or $q_k/b_{k+1}\to 0$
as $k\to\infty$.
Since the local contraction rate equals $mn$ on 
the long interval $(q_k, b_{k+1})$,
it is further easy to see that the right expression
in \eqref{eq:expr} is $mn$, indeed 
\[
\frac{\int_{0}^{b_{k+1}} \delta(q) \; dq }{b_{k+1}} \ge 
\frac{\int_{q_k}^{b_{k+1}} \delta(q) \; dq }{b_{k+1}} = \frac{mn(b_{k+1}-q_k)}{b_{k+1}}= mn(1-o(1)), \qquad k\to\infty,
\]
the reverse inequality being trivial.
Moreover,
by Proposition~\ref{preprep} and \eqref{eq:franzl}
the left expression in \eqref{eq:expr} can be estimated
\begin{align}  \label{eq:insert}
\underline{\delta}(\textbf{f})&\ge \liminf_{k\to\infty} \frac{\int_{q_{k-1}}^{b_k} \delta(q)\; dq + \int_{b_{k}}^{q_k} \delta(q)\; dq }{q_k} \\
&= \liminf_{k\to\infty}\frac{mn(b_k-q_{k-1}) + (mn-m)(q_k-b_k) }{q_k} \nonumber \\
&=mn\cdot (u_k-o(1)) + (mn-m)(1-u_k)  \nonumber \\ 
&=mn(1-o(1))-(1-u_k)m, \nonumber 
\end{align}
as $k\to\infty$, where we have put
\[
u_k:=\frac{b_k}{q_k}.
\]
We determine the limit of $u_k$.
The relations \eqref{eq:iden} and \eqref{eq:frant} imply
\[
f_1(q_k) = \log t_k - q_k/n = q_k\cdot  
\left( \frac{ m^{-1}+n^{-1} }{\tau+1} - \frac{1}{n} + o(1) \right), \quad k\to\infty,
\]
so further by \eqref{eq:preis}, \eqref{eq:frant} 
as $k\to\infty$ we get
\[
b_k = q_k + nf_1(q_k) = q_k \cdot \left(1+ 
\frac{ n/m+1 }{\tau+1} -1 + o(1)\right)= q_k\left( 
\frac{ n/m+1 }{\tau+1} + o(1)\right).
\]
Hence
\[
\lim_{k\to\infty} u_k= \frac{ m+n }{ m(\tau+1) }.
\]
Inserting in \eqref{eq:insert} and a short simplification indeed leads to the claimed lower bound. As observed below Theorem~\ref{thm1}, the reverse upper bound is well-known.

\subsection{Conclusion}

Given $\varepsilon>0$, we have shown
that for large $T$ the set of matrices $T$ close to the constructed template
have Hausdorff and packing dimension at least $\underline{\delta}(\textbf{f})-\varepsilon$ and $\overline{\delta}(\textbf{f})-\varepsilon$, respectively. However,
since the latter values are the predicted dimension of Theorem~\ref{thm1} 
and $\varepsilon$
can be arbitrarily small, the claimed lower bounds follow.

\section{ Proof of Theorem~\ref{thm1} for monotonic $\Phi$}

\subsection{The case $\tau<\infty$}

Now drop the assumption \eqref{eq:1}
and assume instead $\Phi$ is decreasing and $\tau<\infty$.
We again use the construction of~\S~\ref{3.1} but 
will have to specify the involved lacunary sequence $t_k$.
Analyzing the proof of Lemma~\ref{l2}, the only point where we needed
\eqref{eq:1} was to conclude that for the constructed $\textbf{p}_k, \textbf{q}_k$ with 
\[
\Vert \textbf{q}_k\Vert \asymp t_k, \qquad 
\Vert A \textbf{q}_k-\textbf{p}_k \Vert\asymp \Phi(t_k)
\]
for any $\textbf{q}^{\prime}$ of comparable norm as $\textbf{q}_k$, we also have
\[
    \Phi(\Vert \textbf{q}^{\prime}\Vert) \asymp \Phi( \Vert\textbf{q}\Vert).
\]
So to derive the same result for any 
non-increasing approximation function
$\Phi$ with property \eqref{eq:2},
the key is again to prove existence of a sequence $t_k$ where
this holds for similarly derived $\textbf{p}_k, \textbf{q}_k$. This results in
verifying the following rather elementary
lemma on real functions, whose proof however is not as straightforward as it may appear at first sight.

\begin{lemma}  \label{llama}
    Let $\Phi$ be decreasing and inducing $\tau<\infty$, and $c>1$ be fixed. Then there exists
    a sequence of integers $t_k$ satisfying \eqref{eq:aut} and \eqref{eq:lacu} 
    and so that additionally \eqref{eq:1} holds for any $t=t_k$, i.e. there is an explicitly computable fixed $d=d(c,\tau)>1$ independent of $k$ so that
    \begin{equation}  \label{eq:Pr}
    d^{-1}\Phi(t_k)\le \Phi(\tilde{t})\le d\Phi(t_k)
    \end{equation}
    holds for any $k$ and any integer $\tilde{t}$ in the range
    \[
    c^{-1}t_k  \le \tilde{t} \le c t_k.
    \]  
\end{lemma}

The proof of the sole implication \eqref{eq:Pr} would be rather simple, however incorporating
the additional condition \eqref{eq:aut}
makes it slightly technical.
In this context, we point out that \eqref{eq:Pr} is in general not true for an arbitrary sequence $t_k$ satisfying \eqref{eq:aut} and \eqref{eq:lacu} when
we only assume $\Phi$ to be decreasing.
Indeed, there may be problems when $\Phi$ decays fast close to $t_k$. Moreover \eqref{eq:Pr} certainly fails  irrespective of $d$ for all large $t_k$ if for example $\Phi(t)=e^{-t}$ inducing $\tau=\infty$, excluded in the lemma. Indeed we need a different argument in this case. We use both the monotonicity of $\Phi$ and $\tau<\infty$ to show
that we are still able to find a suitable sequence $t_k$ avoiding these obstacles.
As a side note we remark that for the sake of Lemma~\ref{llama} only, we need not assume \eqref{eq:2} on $\Phi$.

\begin{proof}
To prove the lemma, complementing $\tau$, we define the upper order of $1/\Phi$ at infinity
\[
\omega:= \limsup_{t\to\infty} -\frac{\log \Phi(t) }{\log t}.
\]
For simplicity let us assume $\omega<\infty$, otherwise some minor twists in the proof below are required.
Keep in mind that the monotonicity of $\Phi$ 
and $\omega<\infty$ easily imply for any $\sigma>1$ we have
\begin{equation}  \label{eq:fritt}
    \limsup_{N\to\infty} -\left(\frac{ \log \Phi(\sigma N) }{\log (\sigma N)}-\frac{ \log \Phi(N)}{\log N}\right) \le 0, 
\end{equation}
in particular
\begin{equation}  \label{eq:nnp}
    \limsup_{N\to\infty} -\left(\frac{ \log \Phi(N+1) }{\log (N+1)}-\frac{ \log \Phi(N)}{\log N}\right) \le 0. 
\end{equation}
Essentially this means that the function $t\mapsto -\log \Phi(t)/\log t$ can only decrease
very slowly for large $t$ (whereas it may increase arbitrarily rapidly).
We distinguish two cases.

Case 1: $\omega>\tau$, that is the upper
order strictly exceeds the lower order of
$1/\Phi$.
Let $\delta_k\in(0,\omega-\tau)$ be a sequence with $\delta_k\to 0$ slowly. Fix $k$ for the moment. By choice of $\delta_k$ and definition of $\omega$,
there exists an (arbitrarily large) integer $\ell=\ell_k$ satisfying $-\log \Phi(\ell_k)/\log \ell_k > \tau+\delta_k$. 
Hence, by \eqref{eq:nnp}, if $\ell_k$ was chosen large enough,
then by definition of $\tau$
there is some integer $r_k>\ell_k$ with $-\log \Phi(r_k)/\log r_k$ between
$\tau+\delta_k/2$ and $\tau+\delta_k$.
% explain more
Choose $r_k$ minimal with this property.
Again by definition of $\tau$
at some point $s_k>r_k$ we have $-\log \Phi(s_k)/\log s_k < \tau+\delta_k/3$.
By \eqref{eq:fritt} 
for some $\sigma\le c^2$,
we may assume $s_k$ is of the form 
$c^{2g}r_k$ for an integer $g=g(k)\ge 1$, otherwise we slightly enlarge it without
affecting the estimates above.
If we again choose $s_k$ minimal with this property and $\ell_k$ was chosen large enough, again by \eqref{eq:fritt}
we can assume conversely $-\log \Phi(s_k)/\log s_k > \tau$.
(Moreover it follows from montonicity of $\Phi$ that
$s_k/r_k\to\infty$ as $k\to\infty$,   however we will not explicitly need this).
So by this argument we find a sequence of integer pairs $r_k, s_k$ with 
\[
r_1<s_1<r_2<s_2<\cdots
\]
%and $s_k/r_k\to\infty$ 
and
\begin{equation}  \label{eq:conzu}
\tau<-\frac{\log \Phi(s_k) }{\log s_k} <
\tau+ \frac{\delta_k}{3}< \tau+ \frac{\delta_k}{2} <
-\frac{\log \Phi(r_k) }{\log r_k}<\tau+\delta_k.
\end{equation}
We want to contradict this if the implication \eqref{eq:Pr} fails for
\[
d= c^{2\omega}< \infty.
\]
For this, we split the interval $[r_k,s_k]$ into $g=g(k)\ge 1$ adjacent subintervals of constant ratio $c^2$, i.e. of the form
\[
[r_k,s_k]= [r_k,c^2r_k]\cup [c^2r_k,c^4r_k]\cup \cdots\cup [c^{2(g-1)}r_k,c^{2g}r_k]
\]
Denote the subintervals by $I_i=[a_i, b_i]$,
for $1\le i\le g$, 
omitting $k$ in the notation, and formally let $a_{g+1}:=b_g=s_k$. Note that $b_{i}=a_{i+1}=c^2 a_i$ for $1\le i\le g$ and 
\begin{equation}  \label{eq:aiaj}
  a_{i+1}=c^2 a_i, \qquad 1\le i\le g. 
\end{equation}
Now assume contrary to \eqref{eq:Pr}
that on each subinterval $I_i$ we have 
\begin{equation} \label{eq:asinu}
\frac{ \max_{t\in I_i}  \Phi(t) }{\min_{t\in I_i} \Phi(t) }=\frac{   \Phi(a_i) }{ \Phi(a_{i+1}) }  > d=c^{2\omega}, \qquad 1\le i\le g.
\end{equation}
We may assume $r_1=a_1(1)\le a_1(k)=r_k$ was chosen
large enough that
\begin{equation}  \label{eq:Aj}
    \Phi(t) > t^{-2\omega}, 
    \qquad t\ge r_1.
\end{equation}
Combination of the above estimates
\eqref{eq:aiaj}, \eqref{eq:asinu}, \eqref{eq:Aj} readily implies
\[
-\frac{\log \Phi(r_k) }{\log r_k} =
-\frac{\log \Phi(a_1) }{\log a_1} <
-\frac{\log \Phi(a_2) }{\log a_2}<\cdots <  -\frac{\log \Phi(a_{g+1}) }{\log a_{g+1}} = -\frac{\log \Phi(s_k) }{\log s_k} 
\]
contradiction to \eqref{eq:conzu}. So for each $k$ there must be some interval $I_j$, $j=j(k)$, where the claim \eqref{eq:Pr} holds. Moreover,
by \eqref{eq:conzu} we have
\[
-\frac{\log \Phi(t) }{\log t}\in [\tau,\tau+\delta_k], \qquad t\in I_j
\]
so since $\delta_k\to 0$ this tends to $\tau$. So if we let $t_k=c a_j$, $j=j(k)$ as above such that $[c^{-1}t_k,ct_k]=I_j$, and increase $\ell_k$
and thus $t_k$ fast enough in each step,
then indeed all claimed properties are satisfied. Finally the case $\omega=\infty$ requires just minor modifications in the argument, 
we leave the details to the reader.

Case 2: $\omega=\tau$. Then clearly it suffices to find any lacunary
sequence $t_k$ with 
property \eqref{eq:Pr}.
Start with large integer $\ell_1>0$. If $\Phi(c^2 \ell_1)> c^{-4\tau} \Phi(\ell_1)$ then by monotonicity 
$c^{-4\tau} \Phi(\ell_1) \le \Phi(t)\le \Phi(\ell_1)$ for each $t\in [\ell_1, c^2\ell_1]$, in  other words
\[
c^{-4\tau}\le \frac{\Phi(t)}{\Phi(\ell_1)}\le \frac{\Phi(t)}{\Phi(c\ell_1)} \leq 1
\]
for such $t$. Note that $t$ lies in $[c^{-1}(c\ell_1), c(c\ell_1)]$, hence
with $c$ and $t_1=c\ell_1$ as in the theorem.% In particular
%\[
%C^{-4\tau}\le \frac{\Phi(C\ell_1)}{\Phi(\ell_1)} \leq 1.
%\]

Hence
we may put $t_1=c\ell_1$ and are done with first step for implied constant $d=c^{4\tau}$,
and then we choose $\ell_2>\ell_1$
large, inducing $t_2>t_1$ with the same argument. So assume opposite $\Phi(c^2 \ell_1)\le  c^{-4\tau} \Phi(\ell_1)$. Then take $\widetilde{\ell_1}=c^2 \ell_1$ instead. Again by the same argument
we can take $t_1=c\widetilde{\ell_1}=c(c^2\ell_1)=c^3\ell_1$ unless 
if $\Phi(c^2 \widetilde{\ell_1})\le c^{-4\tau} \Phi(\widetilde{\ell_1})$.
We repeat this process and see that the argument to find $t_1$ for the constant $d=c^{4\tau}$ can only fail if for any integer $k\ge 1$
we have $\Phi(c^{2k} \ell_1) \le c^{-4\tau k} \Phi(\ell_1)$. However as $k\to \infty$ this would imply the upper order of $\Phi$ is at least $\omega\ge 4\tau k/(2k)= 2\tau>\tau$, contradiction to our assumption of Case 2.
\end{proof}

With the lemma at hand, it is easy to conclude. By Theorem~\ref{dfsu02}, for any $\varepsilon>0$ and some $T= T_0(\varepsilon)$ large enough we have
uniformly in all templates $\boldsymbol{f}$
\begin{equation}  \label{eq:lhset}
\dim_H(  \{ A: \sup_{q>0}\max_{1\le j\le m+n}|f_j(q)-h_j(q)|    \le T \}) \ge \underline{\delta}(\boldsymbol{f}) - \varepsilon,
\end{equation}
where $h_j=h_{j,A}$ are the component functions of the combined graph associated to $A$.
Now 
starting with decreasing $\Phi$ satisfying \eqref{eq:2} and 
$\tau<\infty$, we may apply Lemma~\ref{llama} for any $c>0$ which equips us with some induced lacunary sequence $t_k=t_k(c)$ from which we construct the template $\textbf{f}=\textbf{f}_c$ as in~\S~\ref{3.1}. 
Now we can copy precisely the proof
of Lemma~\ref{l2} % 4.1
for $T$ above, with the sole
twist that we
use Lemma~\ref{llama}
in place of property \eqref{eq:1}
to obtain \eqref{eq:NN}. Since 
any $A$ within the left hand side set of \eqref{eq:lhset} belongs to $Bad(\Phi)$ again and $\varepsilon>0$
is arbitrarily small, the Hausdorff dimension claim follows with Lemma~\ref{l1} again which still holds for the same reasons.
% it is easy to see that for this $T$ we can apply Lemma~\ref{llama} for some large enough induced $c=c(T)$ so that
% for any induced $A$ we have
% (very similar to the proof of
% Lemma~\ref{l2}). Now Lemma~\ref{llama} equips us with a lacunary sequence $t_k$ from which we construct the template $\textbf{f}$
% as in~\S~\ref{3.1}. Then, using Lemma~\ref{llama}, by the same arguments as in Lemma~\ref{l2} again $A\in SBad(\Phi)$ for any $A$ within the left hand side set of \eqref{eq:lhset}. Combined with Lemma~\ref{l1} which holds for the same reasons as we still have \eqref{eq:aut}, \eqref{eq:lacu}, and as $\varepsilon$
% can be arbitrarily small, we get the Hausdorff dimension result of Theorem~\ref{thm1} again if $\tau<\infty$.
Similar arguments apply for packing dimension. 

\begin{remark}
	It is possible to obtain the final deduction
 using a non-uniform version Theorem~\ref{dfsu02}. However, this requires to first prove a uniform variant of Lemma~\ref{llama}, which is more technical.
\end{remark}

\subsection{The case $\tau=\infty$}  \label{infty}

Since as observed above in this case our key Lemma~\ref{llama} above fails, we use a different argument. Let us denote
for clarity the set $Exact(\Phi)=Exact^{m,n}(\Phi)$ for given $m,n$,
and similarly define $W^{m,n}(\Phi)$.
First recall that it then follows from~\cite{resm} 
that for $n=1$ the set
$Exact^{m,1}(\Phi)\subseteq \mathbb{R}^m$ has full packing dimension $m$, in particular 
it is non-empty (since $\Phi$ is monotonic here, the latter weaker claim also follows directly from Jarn\'ik~\cite[Satz~6]{jarnik}).

Using essentially the method by 
Moshchevitin~\cite[Theorem~12]{ngm} we show the following.

\begin{lemma}  \label{lehre}
	Let %either (a) 
	$\Phi$ induce
	$\tau(\Phi)>\theta$. %or (b) $\Phi$ induce $\tau>n/m$ 
	%and be decreasing.
	Let $\textbf{a}\in Exact^{m,1}(\Phi)\subseteq \mathbb{R}^m$ 
	be a column vector. Then
	for almost all 
	matrices $U\in \mathbb{R}^{m\times (n-1)}$ with respect to $(n-1)m$-dimensional Lebesgue measure, the composed matrix $A:=(\textbf{a},U)\in \mathbb{R}^{m\times n}$ with first column $\textbf{a}$ followed by $U$ on the right
	also lies in $Exact^{m,n}(\Phi)\subseteq \mathbb{R}^{m\times n}$.
\end{lemma}

\begin{remark}
	The claim of the lemma applies, upon modifying $\theta$, more generally to extensions
	of arbitrary matrices (in place of a column vector) via 
	addition of columns, by the same argument. Moreover we may
	replace $Exact^{.,.}(\Phi)$ by $Bad^{.,.}(\Phi)$ throughout.
\end{remark}

 \begin{proof}
%Define the approximation function of some
%$V\in \mathbb{R}^{m^{\prime}\times n^{\prime}}$
%by
%\[
%\psi_{V}(t)=\min_{0<\Vert \textbf{q}\Vert\le t } \Vert V\textbf{q}- %\textbf{p}\Vert
%\]
%with minimum over all integer vectors %$\textbf{q}\in\mathbb{Z}^{n^{\prime}}, %\textbf{p}\in\mathbb{Z}^{m^{\prime}}$
%with (maximum) norm of $\textbf{q}$ at most $t$.
Note that 
%$\psi_{\textbf{a}}(t)\ge \psi_A(t)$ 
%holds 
for any matrix $U$ as in the lemma, for any $q\in \mathbb{Z}, \textbf{p}\in\mathbb{Z}^m$ if we let $\textbf{q}=(q,0,0,\ldots,0)\in\mathbb{Z}^n$ we have
\begin{equation} \label{eq:dannn}
\Vert \textbf{q}\Vert= |q|, \qquad
\Vert A\textbf{q}-\textbf{p}\Vert= \Vert \textbf{a} q - \textbf{p}\Vert. 
\end{equation}
In particular $\textbf{a}\in Exact^{m,1}(\Phi)\subseteq W^{m,1}(\Phi)$ implies $A\in W^{m,n}(\Phi)$.
The substantial part is to show conversely
that $A\notin W^{m,n}(\Phi)/c$ for 
any $c>1$ and Lebesgue almost all $U$. Note that 
$\textbf{a}\notin W^{m,1}(\Phi)/c$ by assumption. For this we use
the method in~\cite{ngm}. %First we show (a).
Let $c>1$ fixed.
For $\textbf{p}\in \mathbb{Z}^m, \textbf{q}\in\mathbb{Z}^n$ 
and writing $\textbf{a}=(a_1,\ldots,a_m)$,
define for $1\le j\le m$ the sets 
\[
J_{N}(p_j,q_1)= (-p_j+q_1 a_{j}-\Phi(N)/c,-p_j+q_1 a_{j})\subseteq \mathbb{R},\qquad N=1,2,\ldots.
\]
Derive
\[
\Omega_N(\textbf{p},\textbf{q}) = \{ U\in \mathbb{R}^{m(n-1)}: 
q_2 u_{j,1} + \cdots + q_n u_{j,n-1} \in J_N(p_j,q_1): 1\le j\le m
 \}
\]
where $U=(u_{i,k})$ for $1\le i\le m, 1\le k\le n-1$, are the entries of $U$. Let further
\[
\Omega_N = \bigcup_{\textbf{p},\textbf{q}:\; \Vert \textbf{q}\Vert=N, \; \textbf{q}^{\ast}\ne 0 } \Omega_N(\textbf{p},\textbf{q})
\]
where $\textbf{q}^{\ast}=(q_2,\ldots,q_n)$, so 
the union is taken over integer vectors where not all $q_2,\ldots,q_n$ vanish and with maximum of $|q_1|,\ldots,|q_n|$ equal to $N$. Note that we do not 
need to take into account vectors with $\textbf{q}^{\ast}=0$
since $\textbf{a}\notin W^{m,1}(\Phi)/c$ and \eqref{eq:dannn}.
Then if converse to our assumption $A\in W^{m,n}(\Phi)/c$
for $A=(\textbf{a},U)$, then the matrix $U$ lies in the
limsup set of the $\Omega_N$, that is it belongs to infinitely many $\Omega_N$. But for this to be true for a positive measure set of matrices $U$, 
by Borel-Cantelli lemma we require
\[
\sum_{N=1}^{\infty} \mu(\Omega_N) = \infty,
\]
where $\mu$ is Lebesgue measure in $\mathbb{R}^{m(n-1)}$.
However, we show that the sum converges to finish the argument.
For any line vector $(u_{j,1},\ldots,u_{j,n-1})$ of an
element in $\Omega_N(\textbf{p}, \textbf{q})$ for given $\textbf{p}, \textbf{q}$,
its distance from the hyperplanes
\[
\{ (x_1,\ldots,x_{n-1}): q_2 x_1 +\cdots+ q_n x_{n-1}= -p_j+q_1a_1 \}\subseteq \mathbb{R}^{n-1}, \qquad 1\le j\le m,
\]
is at most $(\Phi(N)/c)\cdot (q_2^2+\cdots+q_n^2)^{-1/2}\ll \Phi(N) \Vert \textbf{q}^{\ast}\Vert^{-1}$, so
the body in $\mathbb{R}^{m(n-1)}$ where this is true for $1\le j\le m$ has volume
$\ll \Phi(N)^m \Vert\textbf{q}^{\ast}\Vert^{-m}$. Moreover
we can restrict to $\Vert \textbf{p}\Vert \ll N$ as otherwise
$\Omega_N(\textbf{p}, \textbf{q})=\emptyset$. 
 Thus
similar to~\cite{ngm} we calculate
\begin{align*}
\mu(\Omega_N)&\ll  \Phi(N)^m N^m \sum_{q_1\in \mathbb{Z}: |q_1|\le N} \sum_{q_2,\ldots,q_n\in \mathbb{Z}: 1\le \max |q_i|\le N } \frac{1}{\max_{2\le i\le n} |q_i|^m } \\
&\ll  \Phi(N)^m N^{m+1} \sum_{1\le r\le N  } r^{n-m-2} \\
&\ll \Phi(N)^m N^{ \max\{ 1+m, n \} } (\log N)^{\delta(n,m+1) }
\end{align*}
where $\delta(a,b)=1$ if $a=b$ and $\delta(a,b)=0$ otherwise. 
%For the second estimate we used $\Vert \textbf{q}\Vert= N$.
Now since $\tau>\theta$ we know $\Phi(N)<N^{-\theta-\epsilon}$ for
some $\epsilon>0$ and all large $N\ge N_0$. Hence by definition of $\theta$ the tail of the sum is at most $\sum_{N\ge N_0} N^{-1-m\epsilon} \log N$ which indeed converges.
%
%Now for (b) we can more directly apply~\cite[Theorem~12]{ngm}.
%Again clearly by \eqref{eq:dannn} we have $A\in W^{m,n}(\Phi)$ for
%all $U$ (or induced $A$) as in the theorem. Now
%define the approximation function of some
%$V\in \mathbb{R}^{m^{\prime}\times n^{\prime}}$
%by
%\[
%\psi_{V}(t)=\min_{0<\Vert \textbf{q}\Vert\le t } \Vert V\textbf{q}- %\textbf{p}\Vert
%\]
%with minimum over all integer vectors %$\textbf{q}\in\mathbb{Z}^{n^{\prime}}, %\textbf{p}\in\mathbb{Z}^{m^{\prime}}$
%with (maximum) norm of $\textbf{q}$ at most $t$. Then on the other hand,
%indeed if $\Phi$ decreases then $A\in W^{m,n}(\Phi)/c$ for some 
%$c>1$ implies that $\psi_A(t)< \psi_{\textbf{a}}(t)$ for an unbounded %sequence of values $t$,
%because $\textbf{a}\notin W^{m,1}(\Phi)/c$.
%However, from~\cite[Theorem~12]{ngm} we know that this implies that
%the series
%\[
%\sum_{v\ge 1} \Phi(M_v)^m M_{v+1}^{ \max\{ 1+m, n \} } (\log %M_{v+1})^{\delta(n,m+1) }
%\]
%diverges. But the $M_v$ grow exponentially. Hence again a short %calculation shows this is false as soon as $\tau>n/m$.
\end{proof}

We have noticed above that
$\textbf{a}\in Exact^{m,1}(\Phi)\subseteq \mathbb{R}^m$ exist.
Taking any such $\textbf{a}$, for any $A$ derived as in the lemma
$A\in Exact^{m,n}(\Phi)\subseteq Bad^{m,n}(\Phi)$ as well. The Hausdorff dimension of the set of such matrices $A$ 
is by construction at least $(n-1)m$, for as $W_{\infty}$, so  the Hausdorff dimension claim \eqref{eq:beggin} is proved.
For the packing dimension result \eqref{eq:leggin} we
combine Lemma~\ref{lehre} with the estimate 
\begin{equation}  \label{eq:uppack}
\dim_P(X\times Y) \ge \dim_P(X) + \dim_H(Y)
\end{equation}
for any measurable $X,Y$ due to Tricot~\cite{tricot} with $X=Exact^{m,1}(\Phi)\subseteq \mathbb{R}^m$ and $Y$ the set
of $U$ from Lemma~\ref{lehre}. Then
the claim follows from aforementioned full packing dimension of $X$. 
The proof of the case $\tau=\infty$ is complete as well.

	\section{Proof of Theorem~\ref{mo}}
	
	\subsection{The proof}
	If $\tau=\infty$ both (a), (b) follow from the proof
	in~\S~\ref{infty}, so assume $\tau<\infty$.
	Again we proceed similar to~\S~\ref{infty}. However,
	for (a), instead of using~\cite{resm}, 
	we start directly
	with the set of real vectors $\textbf{a}$ in $\mathbb{R}^{m\times 1}$ in $Exact^{m,1}(\Phi)$ of Hausdorff dimension $(m+1)/(\tau+1)>0$ derived in~\cite{bandi}, which applies since $\tau>\theta\ge (n+1)/m>1/m$ and $\Phi$ is assumed decreasing.
	%Note that for $n\ge 2$ clearly the weaker hypothesis \eqref{eq:200} %suffices,
	%while for $n=1$ the Hausdorff dimension claim becomes just the one %in~\cite{bandi}.
	Then since $\tau>\theta$ 
	%(here again some relaxed hypothesis suffices, see the proof of %Lemma~\ref{lehre} above) 
	again Lemma~\ref{lehre} applies as before and we can add
	$(n-1)m$ by the property 
	\begin{equation}  \label{eq:hhd}
	\dim_H(X\times Y)\ge \dim_H(X) + \dim_H(Y)
	\end{equation}
	for any
	measurable $X,Y$, see Tricot~\cite{tricot} again. 
	   
	    For claim (b),
		for any $\Phi$ inducing large enough $\tau$ (not necessarily monotonic), where $\tau>(5+\sqrt{17})/2$ suffices,
		the full packing dimension result from~\cite{resm} used in~\S~\ref{infty}
		in the vector setting $n=1$ still applies. Hence if $\tau>\theta$
		as well, again by Lemma~\ref{lehre} 
		we get \eqref{eq:222}, \eqref{eq:333}, where for the latter
		we also use \eqref{eq:uppack}. 
		
			\subsection{Verification of Remark~\ref{rema} } \label{RR}
				We finally verify the claim in Remark~\ref{rema}.
			By~\cite{bugmo}, since $\tau>\theta>1$ and $(\ast)$ holds for $\Phi$,
			again for $m=n=1$ the set
			$Exact^{1,1}(\Phi)\subseteq \mathbb{R}$ has full Hausdorff dimension $2/(1+\tau)$. Now take any real number $a\in Exact^{1,1}(\Phi)$ and define the first
			column of $A$ as $\textbf{a}=(a,a,\ldots,a)\in \mathbb{R}^{m\times 1}$. It is easy to see that similarly $\textbf{a}\in Exact^{m,1}(\Phi)$. Indeed,
			if $(q,p)$ is an integer vector realizing 
			a good approximation $p/q$ to $a$, then $(q,p,p,\ldots,p)$
			yields the same approximation quality $|qa-p|=\Vert q\textbf{a}-(p,\ldots,p)\Vert$ by choice of maximum norms,
			hence $\textbf{a}\in W^{m,1}(\Phi)$ is implied
			by $a\in W^{1,1}(\Phi)$. Conversely if for
			$(q,p_1,\ldots,p_m)$ we have $\Vert q\textbf{a}-\textbf{p}\Vert< \Phi(q)/c$ for some $c>1$, then also $|qa-p_1|< \Phi(q)/c$, which shows that if $\textbf{a}\in W^{m,1}(\Phi)/c$ then also
			$a\in W^{1,1}(\Phi)/c$. Since this is not true for any $c>1$ as
			$a\in Exact^{1,1}(\Phi)$, it cannot be true for any $c>1$ for $\textbf{a}$ either. Finally by $\tau>\theta$ using
			our $\textbf{a}$ above we can conclude with Lemma~\ref{lehre} and \eqref{eq:hhd} again.
			 We should notice however that the components of our constructed real matrices
			are not $\mathbb{Q}$-linearly independent.

% THEOREM: Bad + L has Lebesgue measure 0.
%
% NOTE: define Bad above
%
% Proof: Write $Bad= \cup F(m)$, with $F(m)$ sets
% with all partial quotients at most $m$.  
% We have $\dim_P(F(m)) = \dim(F(m))<1$ for all $m$. 
% The identity holds since both are
% up a factor the conformal measure (see Mattila book(?) ...),
% the inequlaity is due to Jarnik [cite], see also for refinements...
% Hence Tricot's estimate gives
%\[
%   \dim( F(m) + L ) \le \dim_P( F(m) ) + \dim(L) < 1, \qquad m\ge 1.
%\]
% Thus writing $meas$ for Lebesgue measure we 
% have $meas(F(m)+L)=0$ for any $m$, and we conclude
%
%\[
%  meas(Bad+L) \leq \sum_{m} meas(F(m)+ L) = \sum 0 = 0 
%\]
% The proof is complete.
%
% NOTE: Proof similar to [ich, arXiv: 2201. ..., thm 4.5] proof

\end{document}